\documentclass[a4paper,12pt]{amsart}

\usepackage[margin=1in]{geometry}

\usepackage{amsthm}
\usepackage{cite}
\usepackage[dvipdfmx]{graphicx}
\usepackage{bm}

\theoremstyle{definition}
\newtheorem{theorem}{Theorem}[section]
\newtheorem{lemma}[theorem]{Lemma}

\newtheorem{observation}[theorem]{Observation} 

\theoremstyle{definition}
\newtheorem{definition}[theorem]{Definition}

\newtheorem{pclaim}[theorem]{Claim}
\newtheorem*{ac}{Acknowledgments} 

\theoremstyle{remark}

\newenvironment{rmenum}{
\begin{enumerate}

}
{\end{enumerate}}

\title[Bipartite Graft I]{Bipartite Graft I: Dulmage-Mendelsohn Decomposition for Combs}
\author{Nanao Kita}
\address{Tokyo University of Science 
2641 Yamazaki, Noda, Chiba, Japan 278-0022}
\email{kita@rs.tus.ac.jp}

\newcommand{\dmo}{\leq^\circ} 
\newcommand{\dm}{\leq}

\newcommand{\distgt}[4]{\lambda(#3, #4; #1, #2)}

\newcommand{\distgtf}[5]{\lambda(#4, #5; #3; #1, #2)}

\newcommand{\parcut}[2]{\delta_{#1}(#2)}
\newcommand{\tdmo}{\preceq^\circ}
\newcommand{\tdm}{\preceq} 
\newcommand{\tcomp}[2]{\mathcal{G}(#1, #2)}
\newcommand{\tpart}[2]{\mathcal{P}(#1, #2)}

\newcommand{\nontriconn}[1]{\mathcal{C}^{*}(#1)}

\newcommand{\pargtpart}[3]{\mathcal{P}(#3; #1, #2)}
\newcommand{\bpargtpart}[4]{\mathcal{P}(#3; #1, #2)|_{#4}}

\newcommand{\parNei}[2]{N_{#1}(#2)}
\newcommand{\comp}[1]{\mathcal{G}(#1)}

\newcommand{\up}[1]{\mathcal{U}(#1)}

\begin{document}
\begin{abstract} 
We provide an analogue of the Dulmage-Mendelsohn decomposition for a class of grafts known as comb-bipartite grafts. 
The Dulmage-Mendelsohn decomposition in matching theory is a classical canonical structure theorem for bipartite graphs.  
The substantial part of this classical theorem resides in bipartite graphs that are factorizable, 
that is, those with a perfect matching. 
Minimum joins in grafts, also known as minimum $T$-joins in graphs, 
 is a generalization of perfect matchings in factorizable graphs. 
Seb\"o revealed in his paper that comb-bipartite grafts form one of the two fundamental classes of grafts that serve as skeletons or building blocks of any grafts.  
Particularly, any bipartite grafts, that is, bipartite counterpart of grafts, can be considered as a recursive combination of comb-bipartite grafts. 
In this paper, we generalize the Dulmage-Mendelsohn decomposition for comb-bipartite grafts. 
We also show for this decomposition a property that is characteristics to grafts using the general Kotzig-Lov\'asz decomposition for grafts, 
which is a known graft analogue of another canonical structure theorem from matching theory. 
This paper is the first from a series of studies regarding bipartite grafts. 
\end{abstract} 
\maketitle

\section{Definitions on Sets and Graphs} 
For standard notation for sets and graphs, we follow Schrijver~\cite{schrijver2003}. 
In this section, we explain exceptions or  nonstandard definitions that we use. 
We consider multigraphs. That is, graphs can possess loops and parallel edges. 
For a graph $G$, we denote its vertex and edge sets by $V(G)$ and $E(G)$, respectively. 
For two vertices $u$ and $v$ in a graph, $uv$ denotes an edge whose ends are $u$ and $v$. 
As usual, we often denote a singleton $\{x\}$ by $x$. 
We often treat a graph as its vertex sets. 
We denote the symmetric difference of two sets $A$ and $B$ by $A\Delta B$. 
That is, $A\Delta B = (A\setminus B) \cup (B\setminus A)$. 

We treat paths and circuits as graphs. 
That is, a path is a connected graph in which every vertex is of degree two or less. 
A circuit is a connected graph in which every vertex is of degree two. 
For a path $P$ and vertices $x, y\in V(P)$, we denote by $xPy$ the subpath of $P$ whose ends are $x$ and $y$. 

Let $G$ be a graph in the following. 
For $X\subseteq V(G)$, an {\em ear} relative to $X$ is  
either a path whose vertices except ends are disjoint from $X$ 
or a circuit whose vertices except for one are disjoint from $X$.  
For simplicity, we often treat ears as a path even in the case it is a circuit. 
For example, the {\em ends} of an ear relative to $X$ are its vertices in $X$, which are possibly identical.

For $X, Y \subseteq V(G)$,  we denote by $E_G[X, Y]$ the set of edges whose ends are individually in $X$ and $Y$. 
The set $E_G[X, V(G)\setminus X]$ is denoted by $\parcut{G}{X}$. 
A {\em neighbor} of $X$ is a vertex from $V(G)\setminus X$ that is adjacent to a vertex in $X$. 
The neighbor set of $X$ is denoted by $\parNei{G}{X}$.

Let $H$ and $I$ be subgraphs of $G$. 
The addition of $H$ and $I$ are denoted by $H+I$. 
Let $F \subseteq E(G)$. 
The graph obtained by adding $F$ to $H$ is denoted by $H + F$. 
We denote by $H. F$ the subgraph of $H$ determined by $F$. 
The subgraph of $H$ induced by $X\subseteq V(G)$ is denoted by $H[X]$. 
The subgraph $H[V(H)\setminus X]$ is denoted by $H - X$.

\section{Classical Dulmage-Mendelsohn Decomposition}

For a graph $G$, a set $M \subseteq E(G)$ of edges is a {\em perfect matching} or {\em $1$-factor}  
if $|\parcut{G}{v}\cap M | = 1$ holds for every $v\in V(G)$. 
A graph is {\em factorizable} if it has $1$-factors. 
An edge $e$ from a factorizable graph is {\em allowed} if there is a $1$-factor that contains $e$.  

Let $G$ be a factorizable graph. 
Vertices $u, v\in V(G)$ are {\em factor-connected} if $G$ has a path between $x$ and $y$ whose edges are allowed.  
We say that $G$ is {\em factor-connected} if every two vertices are factor-connected. 
A {\em factor-connected component} or {\em factor-component} of $G$ is a maximal factor-connected subgraph. 
The set of factor-components of $G$ is denoted by $\comp{G}$. 

A factorizable graph consists of its factor-connected components, which are mutually disjoint, 
and edges joining distinct factor-components, which are non-allowed. 
A set of edges is a $1$-factor of $G$ if and only if it is a union of $1$-factors taken from every factor-component. 
Hence, factor-components can be considered as the fundamental building blocks of a factorizable graph regarding $1$-factors. 

The Dulmage-Mendelsohn decomposition~\cite{dm1958, dm1959, dm1963, lp1986} is a classical structure theorem in matching theory.  
This decomposition characterizes the composition of a factorizable bipartite graph from its factor-component in terms of partial order 
and can be stated as follows. 
\begin{theorem}[Dulmage and Mendlesohn~\cite{dm1958, dm1959, dm1963}; see also Lov\'asz and Plummer~\cite{lp1986}] \label{thm:dm} 
Let $G$ be a factorizable bipartite graph with color classes $A$ and $B$. 
For $C_1, C_2 \in \comp{G}$, let $C_1 \dmo_A C_2$ if $C_1 = C_2$ or $E_G[A\cap V(C_2), B\cap V(C_1)] \neq \emptyset$. 
For $C_1, C_2 \in \comp{G}$, let $C_1 \dm_A C_2$ if there exist $D_1, \ldots, D_k \in \comp{G}$, where $k\ge 1$, 
such that $C_1 = D_1$, $C_2 = D_k$, and $D_1 \dmo_A \cdots \dmo_A D_k$. 
Then, $\dm_A$ is a partial order over $\comp{G}$. 
\end{theorem} 

The poset $(\comp{G}, \dm_A)$ that is proved by Theorem~\ref{thm:dm} is called the {\em Dulmage-Mendelsohn poset} 
or the {\em Dulmage-Mendelsohn decomposition} for the factorizable bipartite graph $G$ regarding the color class $A$.

\section{Grafts and Joins} 
\subsection{Basic Definitions} 
Let $G$ be a graph, and let $T\subseteq V(G)$ be a set of vertices. 
For the pair $(G, T)$, a {\em join} of $(G, T)$ is a set of edges $F \subseteq E(G)$ 
such that $|F \cap \parcut{G}{v} |$ is odd if and only if $v\in T$ holds. 
We call the pair $(G, T)$ a {\em graft} if $|T\cap V(C)|$ is even for every connected component $C$ of $G$. 
It can easily be observed from a parity argument 
that the pair $(G, T)$ has a join, which can be an emptyset, if and only if $(G, T)$ is a graft. 
For a graft, {\em minimum} joins, that is, joins with the minimum number of edges, are typically of interest. 
We denote the number of edges in a minimum join of a graft $(G, T)$ by $\nu(G, T)$. 
For a graft $(G, T)$, we often treat items or properties of $G$ as they are from $(G, T)$. 
For example, we say that $e\in E(G)$ is an edge of $(G, T)$. 
We say that a graft $(G, T)$ is {\em bipartite} if $G$ is bipartite. 
For a subgraph $H$ of $G$ such that $(H, V(H)\cap T)$ is a graft, 
we say that $(H, V(H)\cap T)$ is a {\em subgraft} of $(G, T)$.

Minimum joins in graphs are in fact a generalization of $1$-factors in factorizable graphs.  
\begin{observation} 
Let $G$ be a factorizable graph, and let $M \subseteq E(G)$. 
Then, $M$ is a $1$-factor of $G$ if and only if $M$ is a minimum join of the graft $(G, V(G))$. 
\end{observation}

\subsection{Factor-Connectivity in Grafts} 

Let $(G, T)$ be a graft. 
An edge $e\in E(G)$ is {\em allowed} in $(G, T)$ if there is a minimum join $F$ with $e\in F$. 
We say that vertices $x, y\in V(G)$ are {\em factor-connected} in $(G, T)$ 
if $x$ and $y$ are identical, or there is a path between $x$ and $y$ whose edges are all allowed. 
We say that a graft is {\em factor-connected} if every two vertices are factor-connected. 
A {\em factor-connected component} or {\em factor-component} of a graft is a maximal factor-connected subgraft. 
We denote the set of factor-components of $(G, T)$ by $\tcomp{G}{T}$. 
It is easily confirmed from the definition that 
 a graft consists of its factor-components, which are mutually disjoint, and edges between distinct factor-components, 
which are non-allowed.

\subsection{Distances in Graphs} 

Let $(G, T)$ be a graft. 
For $F\subseteq E(G)$, 
we define $w_F: E(G)\rightarrow \{1, -1\}$ as such that $w_F(e) = -1$ for $e\in F$, whereas $w_F(e) = 1$ for $e\in E(G)\setminus F$. 
For $S\subseteq E(G)$, we define $w_F(S) := \Sigma_{e\in S} w_F(e)$.  
For a subgraph $H$ of $G$, which is typically a path or circuit, 
we define $w_F(H) := w_F(E(H))$. 

Let $x, y\in V(G)$. 
If $x\neq y$, we define $\distgtf{G}{T}{F}{x}{y}$ as the minimum value $w_F(P)$, where $P$ is taken over all path between $x$ and $y$. 
If $x = y$,  $\distgtf{G}{T}{F}{x}{y}$ is defined to be the minimum value $w_F(P)$, 
where $P$ is taken over all circuits that contains $x$ or $y$.  
We also call the value $w_F(P)$ as the $F$-weight of $P$. 
We call a path that attains the value $\distgtf{G}{T}{F}{x}{y}$  an {\em $F$-shortest} path between $x$ and $y$. 
Regarding these three definitions, we often omit ``$F$-'' if the meaning is apparent from the context. 

The following characteristic properties hold for minimum joins.  

\begin{lemma}[see Seb\"o~\cite{DBLP:journals/jct/Sebo90}] 
Let $F \subseteq E(G)$ be a join of a graft $(G, T)$. 
If $C$ is a circuit of $G$, then $F\Delta E(C)$ is also a join of $(G, T)$. 
Accordingly, if $F$ is a minimum join, then $w_F(C) \ge 0$ holds for every circuit $C$; 
consequently,  $\distgtf{G}{T}{F}{x}{x} = 0$ holds for every $x\in V(G)$. 
\end{lemma}

\begin{lemma}[Seb\"o~\cite{DBLP:journals/jct/Sebo90}] 
Let $(G, T)$ be a graft, and let $F_1, F_2 \subseteq E(G)$ be minimum joins of $(G, T)$. 
Then, $\distgtf{G}{T}{F_1}{x}{y} = \distgtf{G}{T}{F_2}{x}{y}$ holds for all $x,y\in V(G)$.  
\end{lemma} 

That is, the above lemma states that distances between two vertices do not depend on the choice of a minimum join.  
Therefore,  $\distgtf{G}{T}{F}{x}{y}$ can be denoted by $\distgt{G}{T}{x}{y}$.  
We usually employ this simplified notation in the remainder of this paper.

The next lemma is easily observed from results provided in Seb\"o~\cite{DBLP:journals/jct/Sebo90}. 
See, e.g., Kita~\cite{kita2017parity} for the derivation of this lemma. 
We use this lemma in  later sections.  

\begin{lemma}[see Kita~\cite{kita2017parity}] \label{lem:elem2nonposi}
If $(G, T)$ is a factor-connected graft, then $\distgt{G}{T}{x}{y} \le 0$ holds for every $x, y\in V(G)$.  
\end{lemma}

\section{General Kotzig-Lov\'asz Decomposition for Grafts} 

In this section, we introduce a structure theorem  known as the {\em general Kotzig-Lov\'asz decomposition} for grafts~\cite{kita2017parity}.  
This decomposition is  uniquely determined for a graft and describes the structure of minimum joins. 
This decomposition has an important part in Section~\ref{sec:attribute}. 

\begin{definition} 
Let $(G, T)$ be a graft. 
For $x, y\in V(G)$, we say that $x \sim_{(G, T)} y$  if $x$ and $y$ are contained in the same factor-component and $\distgt{G}{T}{x}{y} = 0$ holds. 
\end{definition}

\begin{theorem}[Kita~\cite{kita2017parity}] \label{thm:tkl}  
Let $(G, T)$ be a graft.  
Then, $\sim_{(G, T)}$ is an equivalence relation over $V(G)$. 
\end{theorem}  

Under Theorem~\ref{thm:tkl}, we denote the family of equivalence classes of $\sim_{(G, T)}$ by $\tpart{G}{T}$. 
This structure is called the {\em general Kotzig-Lov\'asz decomposition} for grafts.  
For each $C\in\tcomp{G}{T}$, the family of  equivalence classes that share vertices with $C$ is a partition of $V(C)$. 
Hence, we denote by $\pargtpart{G}{T}{C}$ the family of equivalence classes that are contained in $V(C)$. 
The subgraft $(C, T\cap V(C))$ also has its general Kotzig-Lov\'asz decomposition $\tpart{C}{T\cap V(C)}$. 
It is easily confirmed from the definition that $\pargtpart{G}{T}{C}$ is generally a proper refinement of $\tpart{C}{T\cap V(C)}$.

\section{Comb-Bipartite Grafts} 

In this section, we  introduce the concept of {\em comb-bipartite} grafts and some lemmas to be used in the later sections. 
Comb-bipartite grafts are closely related to the concept of comb-critical towers that have been introduced in Seb\"o~\cite{DBLP:journals/jct/Sebo90}. 
This class of towers or grafts is important in describing the structure of general grafts 
because they serve as  skeletons of general grafts. 
Particularly, any bipartite grafts can be recursively decomposed into this class of towers and grafts.

\begin{definition} 
Let $(G, T)$ be a bipartite graft, for which $A$ and $B$ are color classes of $G$. 
We say that $(G, T)$ is a {\em comb-bipartite} graft 
with {\em spine set} $A$ and {\em tooth set} $B$ if $\nu(G, T) = |B|$. 
\end{definition}

The next characterization for comb-bipartite grafts can be easily confirmed. 
\begin{lemma}[see Seb\"o~\cite{DBLP:journals/jct/Sebo90} or Kita~\cite{kita2017parity}] 
Let $(G, T)$ be a bipartite graft with color classes $A$ and $B$.   
Then, the following three statements are equivalent: 
\begin{rmenum} 
\item $(G, T)$ is a comb-bipartite graft with spine set $A$ and tooth set $B$. 
\item There exists a minimum join $F$ of $(G, T)$ with $|F\cap \parcut{G}{v} | = 1$ for every $v\in B$. 
\item   $|F\cap \parcut{G}{v} | = 1$ holds for every minimum join $F$ of $(G, T)$ and every $v\in B$. 
\end{rmenum} 
\end{lemma}

\begin{definition} 
Let $(G, T)$ be a comb-bipartite graft with spine set $A$ and tooth set $B$.  
A path or circuit $P$  is {\em $F$-balanced} if $|\parcut{P}{x}| \ge 2$ implies 
$|\parcut{P}{x} \cap F | = |\parcut{P}{x} \setminus F| = 1$ for every $x\in V(P)$. 
\end{definition} 

The concept of $F$-balanced paths frequently shows up when discussing $F$-shortest paths between comb-bipartite grafts. 
The next four lemmas are easily confirmed, and we use these lemmas in later sections sometimes without explicitly mentioning it.

\begin{lemma} \label{lem:balanced} 
Let $(G, T)$ be a comb-bipartite graft with spine set $A$ and tooth set $B$. 
Let $F$ be a minimum join of $(G, T)$. 
Let $x, y \in V(G)$, and let $P$ be an $F$-balanced path between $x$ and $y$ such that $E(P) \neq \emptyset$. 
Let $e_x$ and $e_y$ be the edges of $P$ that are connected to $x$ and $y$, respectively. 
\begin{rmenum} 
\item If $x, y\in A$ holds, then $w_F(P) = 0$. 
\item Let $x\in A$ and $y\in B$. If $e_y \in F$ holds, then $w_F(P) = -1$. 
If $e_y \not\in F$ holds, then $w_F(P) = 1$. 
\item Let $x, y \in B$. If $e_x, e_y \in F$ holds, then $w_F(P) = -2$. 
If $|\{e_x, e_y\} \cap F| = 1$ holds, then $w_F(P) = 0$. 
If $e_x, e_y\not\in F$ holds, then $w_F(P) = 2$. 
\end{rmenum} 
\end{lemma}

\begin{lemma} \label{lem:combpath} 
Let $(G, T)$ be a comb-bipartite graft with spine set $A$ and tooth set $B$. 
Then, 
\begin{rmenum} 
\item for any $x\in A$ and any $y\in B$, $\distgt{G}{T}{x}{y} \ge -1$; 
\item for any $x\in A$ and any $y\in A$, $\distgt{G}{T}{x}{y} \ge 0$; 
\item for any $x\in B$ and any $y \in B$, $\distgt{G}{T}{x}{y} \ge -2$. 
\end{rmenum} 
In each case, the equality is satisfied by $F$-balanced paths between $x$ and $y$. 
\end{lemma}

The next lemma is easily observed from Lemmas~\ref{lem:elem2nonposi} and \ref{lem:combpath}.

\begin{lemma} \label{lem:incomppath} 
 Let $(G, T)$ be a factor-connected comb-bipartite graft with spine set $A$ and tooth set $B$. 
Then, 
\begin{enumerate} 
\item for any $x\in A$ and any $y\in B$, $\distgt{G}{T}{x}{y} = -1$; 
\item for any $x\in A$ and any $y\in A$, $\distgt{G}{T}{x}{y} = 0$; 
\item for any $x\in B$ and any $y \in B$, $\distgt{G}{T}{x}{y}$ is equal to $0$ or $-2$. 
\end{enumerate} 
\end{lemma} 

The next lemma is easily derived from Lemmas~\ref{lem:combpath} and \ref{lem:incomppath}.

\begin{lemma} \label{lem:combtkl} 
 Let $(G, T)$ be a  comb-bipartite graft with spine set $A$ and tooth set $B$, 
 and let $F$ be a minimum join of $(G, T)$. 
 Let $C\in \tcomp{G}{T}$. 
Then, $V(C)\cap A$ is a member of $\tpart{G}{T}$, whereas  $V(C)\cap B$ may be partitioned into multiple members from $\tpart{G}{T}$. 
For any $x, y\in B$, $x\sim_{(G, T)} y$ holds if and only if $\distgt{G}{T}{x}{y} = 0$, 
whereas $x \sim_{(G, T)} y$ does not hold if and  only if  $\distgt{G}{T}{x}{y} = -2$. 
\end{lemma}

Under Lemma~\ref{lem:combtkl},  
for each $C\in \tcomp{G}$, 
we denote the family of equivalence classes of $\tpart{G}{T}$ that constitute $V(C)\cap B$ 
by $\bpargtpart{G}{T}{C}{B}$.

\section{Extension of Dulmage-Mendelsohn Decomposition} 

From this section onward, we provide and prove new results. 
In this section, we show a canonical partial order over the set of factor-components in comb-bipartite grafts. 
We introduce Lemmas~\ref{lem:circuit} and \ref{lem:order2path} to prove Theorem~\ref{thm:tdm}. 
All results introduced in this section are also used in Sections~\ref{sec:path} and \ref{sec:attribute}. 
The next lemma is classically known and is used everywhere in the remaining part of this paper.

\begin{lemma}[see Seb\"o~\cite{DBLP:journals/jct/Sebo90}] \label{lem:circuit} 
Let $(G, T)$ be a graft, and let $F$ be a minimum join of $(G, T)$. 
If $C$ is a circuit with $w_F(C) = 0$, then $F\Delta E(C)$ is also a minimum join of $(G, T)$.  
Accordingly, every edge of $C$ is allowed. 
\end{lemma}

\begin{definition} 
Let $(G, T)$ be a comb-bipartite graft with spine set $A$ and tooth set $B$. 
For any $H_1$ and $H_2$ from $\tcomp{G}{T}$, 
we say $H_1 \tdmo H_2$ if $H_1$ and $H_2$ are identical or 
if $E[H_2\cap A, H_1\cap B]\neq\emptyset$. 
Furthermore, for any  $H_1$ and $H_2$ from $\tcomp{G}{T}$, 
we say $H_1 \tdm H_2$ if there exist $I_1,\ldots, I_k \in \tcomp{G}{T}$, 
where $k \ge 1$, such that $H_1 = I_1$, $H_2 = H_k$, and $I_i \tdmo I_{i+1}$ for each $i\in \{1,\ldots, k\}\setminus \{k\}$. 
\end{definition}

\begin{definition} 
Let $(G, T)$ be a comb-bipartite graft with spine set $A$ and tooth set $B$.  
Let $C_1, C_2\in\tcomp{G}{T}$ be two distinct factor-components with $C_1 \tdm C_2$. 
Let $D_1,\ldots, D_k$, where $k \ge 2$,  be distinct factor-components such that $D_1 = C_1$, $D_k = C_2$, and 
$D_i \tdmo D_{i+1}$ holds for every $i\in \{1,\ldots, k-1\}$. 
We call such  $D_1,\ldots, D_k$ a {\em defining sequence} for $C_1 \tdm C_2$. 
\end{definition}

The next lemma is derived from Lemmas~\ref{lem:incomppath} and \ref{lem:circuit}.

\begin{lemma} \label{lem:order2path} 
Let $(G, T)$ be a comb-bipartite graft with spine set $A$ and tooth set $B$.  
Let $C_1, C_2\in\tcomp{G}{T}$ be two distinct factor-components with $C_1 \tdm C_2$, and 
let $D_1,\ldots, D_k$, where $k \ge 2$, be a defining sequence for $C_1 \tdm C_2$. 
Let $F$ be a minimum join of $(G, T)$. 
Then, 
\begin{rmenum} 
\item 
for every $x\in V(C_1)\cap A$ and every $y\in V(C_2)\cap A$, there is a path of $F$-weight zero between $x$ and $y$ whose vertices are contained in $V(D_1)\cup \cdots \cup V(D_k)$; and, 
\item 
for every $x\in V(C_1)\cap A$ and every $y\in V(C_2)\cap B$, there is a path of $F$-weight $-1$ between $x$ and $y$ whose vertices are contained in $V(D_1)\cup \cdots \cup V(D_k)$. 
\end{rmenum} 
\end{lemma} 
\begin{proof} %lem:orderpath
We prove the lemma by the induction on $k$. 
The statement obviously holds for the case where $k = 1$.  
Now, let $k > 1$, and assume that the statement holds for every case where $k$ is less. 
By definition, $D_1 \tdm D_{k-1}$ holds, for which $D_1, \ldots, D_{k-1}$ are a defining sequence.  
Let $e\in E_G[A\cap V(D_k), B\cap V(D_{k-1})]$, and let $u\in A\cap D_k$ and $v\in B\cap D_{k-1}$ be the ends of $e$. 
Let $x\in V(D_1)\cap A$ and $y\in V(D_k)$.   
The induction hypothesis implies that there is a path $P$ between $x$ and $v$ with $w_F(P) = -1$ 
and $V(P)\subseteq V(D_1)\cup\cdots V(D_{k-1})$. 
By contrast, Lemma~\ref{lem:elem2nonposi} implies that $D_k$ has a path $Q$ between $u$ and $y$ 
such that $w_F(Q) = 0$ or $w_F(Q) = -1$ for the cases where $y\in A$ or $y\in B$, respectively. 
Hence, $P + Q$ is a path between $x$ and $y$ that proves the statement for $k$.  
Thus, the lemma is proved. 
\end{proof}

Lemmas~\ref{lem:circuit} and \ref{lem:order2path} derive Theorem~\ref{thm:tdm}.

\begin{theorem}  \label{thm:tdm} 
If $(G, T)$ is a comb-bipartite graft with spine set $A$ and tooth set $B$, 
then $\tdm$ is a partial order over $\tcomp{G}{T}$. 
\end{theorem} 
\begin{proof} %thm:tdm
Because reflexivity and transitivity are obvious from the definition,  we prove antisymmetry in the following. 
Suppose, to the contrary, that $H_1$ and $H_2$ are two distinct factor-components with $H_1\tdm H_2$ and $H_2 \tdm H_1$. 
Let $I_1,\ldots, I_k \in \tcomp{G}{T}$, where $k \ge 2$, be a defining sequence for $H_1 \tdm H_2$. 
Let $I_k,\ldots, I_l \in \tcomp{G}{T}$, where $l > k$, be a defining sequence for $H_2 \tdm H_1$. 
Let $F$ be a minimum join of $(G, T)$. 
By this definition, 
 there exist $p$ and $q$ with $1 \le p < q \le l$ and $|p - q| > 1$ 
such that, among $I_p,\ldots, I_q$, only $I_p$ and $I_q$ are pairwise identical. 
Let $x\in V(I_{q-1})\cap B$ and $y\in V(I_p)\cap A$ be the vertices with $xy\in E(G)$. 
By Lemma~\ref{lem:order2path}, 
there is a path $P$ between $x$ and $y$ with  $w(P) = -1$ and $V(P)\subseteq V(I_p)\dot\cup\cdots\dot\cup V(I_{q-1})$. 
Then, $P + xy$ is  a circuit with $w_F(P + wy) = 0$.   
This implies from Lemma~\ref{lem:circuit} that $xy$ is an allowed edge of $(G, T)$, 
which is a contradiction. 
The proof is completed. 
\end{proof}

That is, Theorem~\ref{thm:tdm} states that $(\tcomp{G}{T}, \tdm)$ is a poset. 
We call this poset the {\em Dulmage-Mendelsohn poset} of the comb-bipartite graft $(G, T)$, 
or the {\em Dulmage-Mendelsohn decomposition} when we refer to it as a decomposition of a graft. 
For $C\in \tcomp{G}{T}$,  
we denote the set of strict upper bounds of $C$ by $\up{C}$.

\section{Structure of Paths} \label{sec:path}

In this section, we introduce new properties on the structure of paths and ears in comb-bipartite grafts to be used in Section~\ref{sec:attribute}. 
The goal of this section is to prove Lemma~\ref{lem:ear2sim}, which is used in Section~\ref{sec:attribute}, and 
we provide three lemmas, Lemmas~\ref{lem:noheteroear}, \ref{lem:relativepath}, and \ref{lem:ear2disjoint}, for deriving Lemma~\ref{lem:ear2sim}.

The next lemma is derived from Lemmas~\ref{lem:incomppath} and \ref{lem:circuit}.

\begin{lemma} \label{lem:noheteroear} 
Let $(G, T)$ be a comb-bipartite graft with spine set $A$ and tooth set $B$. 
Let $F$ be a minimum join of $(G, T)$. 
If $P$ is an $F$-balanced ear relative to $C \in \tcomp{G}{T}$, then the ends of $P$ are in $B\cap V(C)$ 
and, accordingly, $w_F(P) = 2$.  
\end{lemma} 
\begin{proof} %lem:noheteroear
Let $x, y$ be the ends of $P$. 
First, suppose  $x\in A\cap V(C)$ and $y \in B \cap V(C)$. 
Then, Lemma~\ref{lem:balanced} implies $w_F(P) = 1$. 
However, Lemma~\ref{lem:incomppath} implies that $C$ has a path $Q$ between $x$ and $y$ with $w_F(Q) = -1$. 
Thus, $P + Q$ is a circuit of weight zero that contains non-allowed edges. This contradicts Lemma~\ref{lem:circuit}.  
The case where the both ends of $P$ are in $A\cap V(C)$  lead to a contradiction by a similar discussion. 
Thus, we obtain $x, y \in B\cap V(C)$.  
Lemma~\ref{lem:balanced} now implies $w_F(P) = 2$. 
\end{proof}

We now introduce a new notation. 
For a graph $H$, we denote by $\nontriconn{H}$ the set of connected components with more than one vertex. 

In the next two lemmas, note the following observation. 
Let $P$ be a path in a graph $G$ with $E(P)\neq \emptyset$, and let $F \subseteq E(G)$. 
Then, $\nontriconn{P. F}$ and $\nontriconn{P -F}$ are sets of edge disjoint paths in which each paths has one edge or more.  
Also, if we trace $P$ from one end, then paths from $\nontriconn{P. F}$ and $\nontriconn{P -F}$ appear alternately on $P$.

Lemma~\ref{lem:noheteroear} implies the next lemma.

\begin{lemma} \label{lem:relativepath} 
Let $(G, T)$ be a comb-bipartite graft with spine set $A$ and tooth set $B$. 
Let $F$ be a minimum join of $(G, T)$. 
Let $C\in \tcomp{G}{T}$. 
Let $x, y\in V(C)$, and let $P$ be a path beween $x$ and $y$ such that $w_F(P) = -2$. 
Then, 
\begin{rmenum} 
\item \label{item:relativepath:ear}  each connected component from $\nontriconn{P - E(C)}$ is an ear relative to $C$ whose $F$-weight is $2$; and, 
\item \label{item:relativepath:inpath} each connected component from $\nontriconn{P. E(C)}$ is  a path whose $F$-weight is $-2$. 
\end{rmenum}   
\end{lemma} 
\begin{proof} %lem:relativepath
Obviously, $E(P)\neq \emptyset$, and Lemma~\ref{lem:combpath} implies that $P$ is an $F$-balanced path with $x, y \in B$. First, suppose that a path $Q \in \nontriconn{ P - E(C)}$ is not an ear relative to $C$; 
that is, $Q$ has an internal vertex in $V(C)$. 
Let $x$ be an end of $Q$,  trace $Q$ from $x$, and let $v$ be the first encountered vertex in $V(C)$. 
Then, $xQv$ is an $F$-balanced ear relative to $C$. Hence, Lemma~\ref{lem:noheteroear} implies $v\in B$. 
However, because $v$ is an internal vertex of $Q$, we have $\parcut{P}{v} \cap F = \emptyset$.   
This is a contradiction because $v\in B$ is an internal vertex of an $F$-balanced path $P$. 
Hence, every $Q \in \nontriconn{ P - E(C)}$ is an $F$-balanced ear relative to $C$, and Lemma~\ref{lem:noheteroear} accordingly implies $w_F(Q) = 2$. 
Thus, \ref{item:relativepath:ear} is proved. 

It now follows from \ref{item:relativepath:ear} that each path from $\nontriconn{ P. E(C) }$ is an $F$-balanced path both of whose ends are in $B$  
and are connected to edges from $F$. 
Therefore, Lemma~\ref{lem:balanced} now proves \ref{item:relativepath:inpath}. 
This completes the proof.

\end{proof}

Lemmas~\ref{lem:noheteroear} and \ref{lem:relativepath}, together with Lemma~\ref{lem:circuit}, derive the next lemma.

\begin{lemma} \label{lem:ear2disjoint} 
Let $(G, T)$ be a comb-bipartite graft with spine set $A$ and tooth set $B$. 
Let $F$ be a minimum join of $(G, T)$. 
Let $C \in \tpart{G}{T}$, and let $s, t\in V(C)\cap B$ be vertices with $s\not\sim_{(G, T)} t$.  
Let $P$ be a  path between $s$ and $t$ such that $w_F(P) = -2$, and let $Q$ be an $F$-balanced path one of whose ends is $t$. 
If $Q$ does not have any vertices in $V(C)\setminus \{t\}$, then $Q$ is disjoint from $P - V(C)$. 
\end{lemma} 
\begin{proof} %lem:ear2disjoint 
Suppose that $Q$ shares a vertex with $P - V(C)$.  
Trace $Q$ from $t$, and let $x$ be the first encountered vertex in $P - V(C)$. 
In the following, note that, 
 for each $L \in \nontriconn{P - E(C)}$, the ends of $L$ are connected to non-allowed edges in $\parcut{G}{C}$; 
 by contrast,  
 for each $L \in \nontriconn{P. E(C) }$,  the ends of $L$ are connected to edges from $F$.

\begin{pclaim} \label{claim:ear2disjoint:x2B} 
It holds that $w_F(tQx) = 2$ and $w_F(xPt) = 0$. 
\end{pclaim} 
\begin{proof} 
We first prove that $x$ is in $B$. 
Suppose $x\in A$. 
Then, Lemma~\ref{lem:balanced} implies $w_F(tQx) = 1$ and $w_F(xPt) =  -1$.  
It follows that $tQx + xPt$ is a circuit of weight zero that contains non-allowed edges, which contradicts Lemma~\ref{lem:circuit}. 
Thus, we obtain $x\in B$.  

Therefore, Lemma~\ref{lem:balanced} further implies $w_F(tQx) \in \{0, 2\}$ and $w_F(xPt) \in \{-2, 0\}$.  
From Lemma~\ref{lem:circuit} again,  we obtain $w_F(tQx) = 2$ and $w_F(xPt) = 0$.  
The claim is proved.

\end{proof} 

Trace $sPx$ from $x$, and let $y$ be the first encountered vertex in $C$. 
Additionally, trace $xPt$ from $x$, and let $z$ be the first encountered vertex in $C$; 
note  $yPz \in \nontriconn{P- E(C)}$.  
Lemma~\ref{lem:relativepath} implies  $w_F(yPz) = 2$. 
Note also $w_F(zPt) = -2$.

\begin{pclaim} \label{claim:ear2disjoint:nonneg} 
 It holds that $w_F(xPy) = 0$ and $w_F(xPz) = 2$. 
\end{pclaim} 
\begin{proof} 
 Claim~\ref{claim:ear2disjoint:x2B} implies that 
$w_F(xPz) =  w_F(xPt) - w_F(zPt) = 0 - (-2) = 2$. This further implies $w_F(xPy) = w_F(yPz) - w_F(xPz) = 2 - 2 = 0$. 
The claim is proved. 
\end{proof} 

Let $P'$ be the path from $\nontriconn{P. E(C)}$ that contains $z$, and let $u$ be the end of $P'$ other than $z$. 
Note that Lemma~\ref{lem:relativepath} implies $y, z, u\in B$.  
Lemmas~\ref{lem:combtkl} and \ref{lem:relativepath} \ref{item:relativepath:inpath} imply 
that $z$ and $u$ are contained in the same member of $\tpart{C}{T\cap V(C)}$. 
Hence, either $y\not\sim_{(C, T\cap V(C))} z$ or $y \not\sim_{(C, T\cap V(C))} u$ holds.  
Therefore, there is a path $R$ in $C$ between $y$ and $z$ or $u$ such that $w_F(R) = -2$. 
Trace $R$ from $y$, and let $v$ be the first encountered vertex in $yPt$. 
Note that $\parcut{R}{y} \subseteq F$ holds, according to Lemma~\ref{lem:balanced}.

\begin{pclaim} \label{claim:ear2disjoint:inseg} 
It holds that $w_F(yRv) = 0$ and $w_F(vPt) = -2$. 
\end{pclaim}  
\begin{proof} 
First, we prove $v\in B$. 
If $v\in A$ holds, then Lemma~\ref{lem:balanced} implies $w_F(yRv) = -1$ and $w_F(vPy) = 1$; 
accordingly, $yRv + vPy$ is a circuit of weight zero that contains non-allowed edges, which contradicts Lemma~\ref{lem:circuit}. 
Hence,  $v\in B$ holds.    

This further implies from Lemma~\ref{lem:balanced} that $w_F(yRv) \in \{-2, 0\}$ and $w_F(vPy) \in \{0, 2\}$. 
From Lemma~\ref{lem:circuit} again, we have $w_F(yRv) = 0$ and $w_F(vPy) = 2$. 
This further implies $w_F(vPt) = w_F(zPt) - w_F(zPv) = -2 - 0 = -2$. 
The claim is proved. 

\end{proof} 

It now follows from Claims~\ref{claim:ear2disjoint:x2B}, \ref{claim:ear2disjoint:nonneg}, and \ref{claim:ear2disjoint:inseg}  that 
$tQx + xPy + yRv + vPt$ is a circuit whose weight is equal to $w_F(tQx) + w_F(xPy) + w_F(yRv) + w_F(vPt) = 2 + 0 + 0  -2 = 0$
 that contains non-allowed edges; this again contradicts Lemma~\ref{lem:circuit}. 
Thus, the lemma is proved. 
\end{proof}

Lemmas~\ref{lem:circuit} and \ref{lem:ear2disjoint} imply the next lemma.

\begin{lemma}  \label{lem:ear2sim} 
Let $(G, T)$ be a comb-bipartite graft with spine set $A$ and tooth set $B$. 
Let $F$ be a minimum join of $(G, T)$. 
Let $C \in \tpart{G}{T}$, and let $P$ be an $F$-balanced ear relative to $C$, whose ends are $s, t\in V(C)$. 
Then, $s, t \in V(C)\cap B$ and $s \sim_{(G, T)} t$ hold. 
\end{lemma} 
\begin{proof} 
Lemma~\ref{lem:noheteroear} implies $s, t\in V(C)\cap B$ and $w_F(P) = 2$. 
Suppose $s \not \sim_{(G, T)} t$.    
Then, there exists  a path  $Q$ between $s$ and $t$ with $w_F(Q) = -2$. 
Lemma~\ref{lem:ear2disjoint} implies that $P - s - t$ is disjoint from $Q$. 
Hence, $P + Q$ is a circuit of weight zero that contains non-allowed edges, which contradicts Lemma~\ref{lem:circuit}. 
\end{proof}

\section{Attributes of Upper Bounds} \label{sec:attribute}

We now show a structural property regarding the partial order $\tdm$ using the general Kotzig-Lov\'asz decomposition. 
We prove that,  for each factor-component $C$ of a comb-bipartite graft $(G, T)$,   
each upper bound regarding $\tdm$ has a ``label'' chosen from the members of $\pargtpart{G}{T}{C}$ that are contained in the tooth set. 
In the following, we provide and prove Lemmas~\ref{lem:attribute} and \ref{lem:nodiamond}, 
and these two lemmas immediately derive Theorem~\ref{thm:attribute}, which states the existence of those ``labels''.

The next lemma is derived from Lemmas~\ref{lem:order2path} and \ref{lem:ear2sim}. 

\begin{lemma} \label{lem:attribute} 
Let $(G, T)$ be a comb-bipartite graft with spine set $A$ and tooth set $B$.  
Let $C_1, C_2\in\tcomp{G}{T}$ be two distinct factor-components with $C_1 \tdm C_2$. 
Let $D_1,\ldots, D_k$, where $k \ge 2$,  be a defining sequence for $C_1 \tdm C_2$. 
Then,  there  exists $S\in \tpart{G}{T}$ with $S \subseteq V(C_1) \cap B$  such that $\parNei{G}{D_i} \cap V(C_1) \subseteq S$ 
holds for every $i\in \{2,\ldots, k\}$.   
\end{lemma} 
\begin{proof} %lem:attribute
Assume that there exist $i, j \in \{2,\ldots, k\}$ with $i < j$ such that 
$\parNei{G}{D_\alpha} \cap V(C_1) \neq \emptyset$ for each $\alpha \in \{ i, j\}$.  
For each $\alpha\in\{i, j\}$, let $u_\alpha \in A\cap V(D_\alpha)$ and $v_\alpha \in B\cap V(C_1)$ be vertices with $u_\alpha v_\alpha \in E(G)$.

Let $F$ be a minimum join of $(G, T)$.  

Because $D_i \tdm D_j$ holds, 
Lemma~\ref{lem:order2path} implies that there is a path $P$ between $u_i$ and $u_j$ such that $w_F(P) = 0$ and $V(P)\subseteq V(D_i)\cup \cdots \cup V(D_j)$.    
Thus, $P + u_1v_1 + u_2v_2$ is an $F$-balanced ear relative to $C_1$. 
Therefore, Lemma~\ref{lem:ear2sim} implies $v_i \sim v_j$. 
This completes the proof. 
\end{proof}

Under Lemma~\ref{lem:attribute}, we provide another property of the partial order $\tdm$ using Lemmas~\ref{lem:incomppath}, \ref{lem:order2path}, and \ref{lem:ear2sim}.

\begin{lemma} \label{lem:nodiamond} 
Let $(G, T)$ be a comb-bipartite graft with spine set $A$ and tooth set $B$.  
Let $C_0\in \tcomp{G}{T}$, and let $C_1, C_2 \in \tcomp{G}{T} \setminus \{C_0\}$ be factor-components 
such that $C_0 \tdmo C_1$ and $C_0 \tdmo C_2$. 
Let $S_1, S_2 \in \tpart{G}{T}$ be  equivalence classes such that $\parNei{G}{C_i} \cap V(C_0) \subseteq S_i$ for each $i\in\{1, 2\}$.   
If $S_1$ and $S_2$ are distinct, then no factor-component $C$ satisfies $C_1 \tdm C$ and $C_2 \tdm C$ at the same time.  
\end{lemma} 
\begin{proof} %lem:nodiamond 
Suppose, to the contrary, that $C_3\in\tcomp{G}{T}$ satisfies both $C_1 \tdm C_3$ and $C_2 \tdm C_3$.    
For each $i\in \{1, 2\}$, let $s_i \in S_i$ and $t_i\in A\cap V(C_i)$ be vertices with $s_it_i \in E(G)$. 
Under Lemma~\ref{lem:attribute}, we can assume that $C_1, C_2, C_3$ are mutually distinct. 
For each $i\in \{1, 2\}$, 
let $H^i_1,\ldots, H^i_{k_i} \in \tcomp{G}{T}$, where $k_i \ge 2$, be a defining sequence for $C_i \tdm C_3$. 
We can assume 
that $H^1_i$ and $H^2_j$ are distinct for any $i\in \{1,\ldots, k-1\}$ and any $j\in \{1, \ldots, l-1\}$.

Let $F$ be a minimum join of $(G, T)$. 
For each $i\in \{1, 2\}$,   
Lemma~\ref{lem:order2path} implies that  
there exists a path $P_i$ of weight zero between $t_i$ and a vertex $u_i \in A\cap V(C_3)$ 
whose vertices except $u_i$ are $V(H^i_1)\cup \cdots \cup V(H^i_{k_i -1})$. 
Additionally, Lemma~\ref{lem:incomppath} implies that 
$C_3$ has a path $Q$ of weight zero between $s_1$ and $s_2$. 
Then, $P_1  + Q + P_2 + s_1t_1 + s_2t_2$  is an $F$-balanced ear relative to $C_0$. 
Hence, Lemma~\ref{lem:ear2sim} proves that $S_1 = S_2$. 

The lemma is proved. 
\end{proof}

Lemmas~\ref{lem:attribute} and \ref{lem:nodiamond} immediately imply the following theorem. 

\begin{theorem} \label{thm:attribute} 
Let $(G, T)$ be a comb-bipartite graft with spine set $A$ and tooth set $B$.  
For each $C_0 \in \tcomp{G}{T}$, 
there uniquely exists a partition $\{ \up{S} \subseteq \up{C_0}: S\in \bpargtpart{G}{T}{C}{B} \}$ of $\up{C_0}$, in which some members can be empty, that satisfies the following two properties. 
\begin{rmenum}  
\item 
If $D\in\up{C_0}$ is a factor-component with $\parNei{G}{D}\cap V(C_0)\neq \emptyset$ and 
$S \in \bpargtpart{G}{T}{C}{B}$ is the equivalence class with $\parNei{G}{D}\cap V(C_0) \subseteq S$, 
then $D\in\up{S}$ holds. 
\item 
If a factor-component $D \in \up{C_0}$ satisfies $E_G[D, D'] \neq \emptyset$ for a factor-component $D'\in \up{S}$, 
then  $D\in \up{S}$ holds. 

\end{rmenum} 
\end{theorem}

Under Theorem~\ref{thm:attribute}, we can define the attributes regarding the partial order $\tdm$.  
That is, we say that the attribute of $D\in \up{C_0}$ is $S \in \bpargtpart{G}{T}{C_0}{B}$ 
if $D\in \up{S}$ holds, where $\up{S}$ is the member of the partition of $\up{C_0}$ 
as provided in Theorem~\ref{thm:attribute}.

\begin{ac} 
This study was supported by JSPS KAKENHI Grant Number 18K13451. 
\end{ac}

\bibliographystyle{splncs03.bst}
\bibliography{tdm.bib}

\end{document}